\documentclass[1p,preprint, 10pt]{elsarticle}
\usepackage{amsfonts}
\usepackage{amsmath}
\usepackage{amsthm}
\usepackage{amssymb}
\usepackage{url}
\usepackage{hyperref}

\usepackage{comment}

\newcommand{\calA}{\mathcal{A}}
\newcommand{\calB}{\mathcal{B}}

\newcommand{\calT}{\mathcal{T}}

\newcommand{\bbR}{\mathbb{R}}
\newcommand{\bfz}{\mathbf{z}}
\newcommand{\bmx}{\begin{bmatrix}}
\newcommand{\emx}{\end{bmatrix}}
\newcommand{\setX}{\mathcal{X}}
\newcommand{\setY}{\mathcal{Y}}
\newcommand{\setP}{\mathcal{P}}
\newcommand{\setA}{\mathcal{A}}
\newcommand{\vdm}{\mathop{\mathrm{vdm}}}
\newcommand{\VDM}{\mathop{\mathrm{VDM}}}
\newcommand{\diag}{\mathop{\mathrm{diag}}}

\newtheorem{proposition}{Proposition}
\newtheorem{corollary}{Corollary}

\begin{document}

\begin{frontmatter}
\title{On certain multivariate Vandermonde determinants whose variables separate}
\author[padua]{Stefano De Marchi}
\ead{demarchi{@}math.unipd.it}
\author[vub]{Konstantin Usevich\corref{cor1}}
\ead{Konstantin.Usevich{@}vub.ac.be}

\cortext[cor1]{Corresponding author}

\address[padua]{University of Padova, Department of Mathematics,
Via Trieste, 63, I-35121 PADOVA, Italy}

\address[vub]{Vrije Universiteit Brussel, Department ELEC, Pleinlaan 2, B-1050, Brussels, Belgium}

\begin{abstract}
We prove that for almost square tensor product grids and certain sets of bivariate polynomials the Vandermonde determinant can be factored into a product of univariate Vandermonde determinants. This result generalizes the conjecture \cite[Lemma 1]{Bos.2009.OTV}. As a special case, we apply the result to Padua and Padua-like points.
\end{abstract}


\end{frontmatter}

\section{Introduction}
This paper is mainly inspired by the results in \cite{Bos.2009.OTV} where the authors discussed the properties of the Vandermonde determinants associated to point sets on the square that distribute as the Padua points. In order to understand the result of this article, we briefly recall the definition
and the construction of the Padua points.

The Padua points are the first known near-optimal point set for bivariate polynomial interpolation of total degree in the square $[-1,1]^2$, whose Lebesgue constants have minimal order of growth of ${\cal O}((\log n)^2)$, $n$ being the polynomial degree \cite{Bos.2006.BLI,  Bos.2007.BLI}.

It has been observed that these points have the structure of the union of {\it two} (tensor product) grids of Chebyshev-Lobatto points, 
one square and the other rectangular. Actually there are four families of Padua
points, obtainable one from the other by a suitable rotation of 90, 180 or 270
degrees. For the sake of simplicity, we consider here only the construction of the points belonging to the first family, displayed in Fig.~\ref{fig1}.

Let start by taking the $n+1$ Chebyshev--Lobatto points on $[-1,1]$
\begin{equation*}
C_{n+1}:=\left\{z^n_j=\cos\left(\frac{(j-1)\pi}{n}\right),\ j=1,\ldots,n+1\right\}\,.
\end{equation*}
We then consider two subsets of points with {\it odd} and {\it even} indices
\begin{equation*}
\begin{aligned}
C_{n+1}^\mathrm{o}&:= \left\{z^n_j,\ j=1,\ldots,n+1,
\ j\ \text{odd}\right\}\\
C_{n+1}^\mathrm{e}&:= \left\{z^n_j,\ j=1,\ldots,n+1, \ j\
\text{even}\right\}
\end{aligned}
\end{equation*}
Then, the Padua points of the first family are the set
\begin{equation} \label{pd1}
{\cal P}_n := \left( C_{n+1}^\mathrm{o}\times
C_{n+2}^\mathrm{o} \right)  \cup \left( C_{n+1}^\mathrm{e}\times
C_{n+2}^\mathrm{e} \right) \subset C_{n+1}\times C_{n+2}\,.
\end{equation}
These points have cardinality of the space of bivariate polynomials of degree $\le n$, i.e. $N=(n+1)(n+2)/2$.

There is another interesting geometric interpretation: Padua points are self-intersections and boundary
contacts of the following (parametric and periodic) generating curve
\begin{equation*}
\gamma(t)=(-\cos((n+1)t),-\cos(nt)),\quad t\in[0,\pi]
\end{equation*}
which turns out to be a {\it Lissajous curve} \cite{Bos.2006.BLI}. In Fig.~\ref{fig1}, we show the two grids and the generating curve for $n=4$. In this case, the square grid has 9 points while the rectangular one has 6 points.
\begin{figure}[ht!]
\centering
\includegraphics[height=5cm]{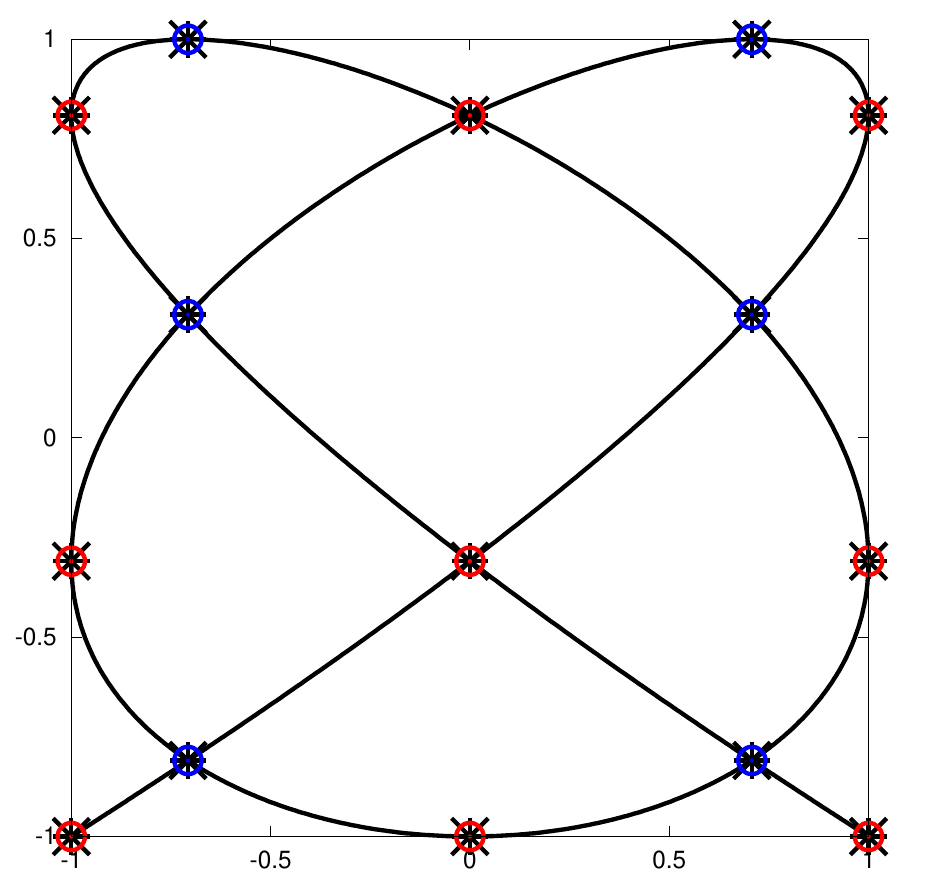}
\caption{Padua points and their generating curve for $n=4$. The grids of odd and even indices are indicated with different colours and style. } \label{fig1}
\end{figure}

For more details on Padua points, their properties and applications we refer the 
interested reader to the web page 
\href{http://www.math.unipd.it/~marcov/CAApadua.html}{http://www.math.unipd.it/$\sim$marcov/CAApadua.html} 
that also contains an up-to-date bibliography on the topic.

In \cite{Bos.2009.OTV} has been conjectured an interesting formula for the Vandermonde determinant of {\it any set of points} with exactly a similar  
structure like that of Padua points. Surprisingly, the Vandermonde determinant factors into the product of two univariate functions. 
The technical Lemma 1 \cite[Lemma 1]{Bos.2009.OTV}, very important in that paper, was conjectured
to be true but up to now a correct proof has not been given. This article provides a general proof of this (special) 
factorization that applies to any set of points having a similar structure. 

\section{Notation}
We denote by $\bbR^{m \times n}$ the space of $m\times n$ real matrices, by $\diag(V) \in \bbR^{n\times n}$ 
the diagonal matrix constructed from $V \in \bbR^{n}$, and by $I_n \in \bbR^{n\times n}$ the identity matrix.
We denote by $\bbR^{m\times n}[\bfz]$ the space of $m\times n $ real polynomial matrices, in $q$ variables $\bfz = (z_1, \ldots, z_q)$. For a polynomial matrix $P \in \bbR^{m\times n}[\bfz]$, we denote by $P_{:,j} \in \bbR^{m\times 1}[\bfz]$ (resp. $P_{i,:} \in \bbR^{1\times n}[\bfz]$) the  $j$-th column (resp. $i$-th row), and by $P_{:,j:k} \in \bbR^{m\times (k-j+1)}[\bfz]$ the submatrix constructed from the $j$-th to $k$-th columns of $P$. For two univariate polynomial matrices $P \in \bbR^{m\times n}[x]$ and $Q \in \bbR^{m\times n}[y]$, we denote their Hadamard (element-wise) product as $P \circ Q \in \bbR^{m\times n}[\bfz]$, $\bfz=(x,y)$. 

For a set of polynomials $\setP = \{p_1(\bfz), \ldots, p_n(\bfz) \}$ and a set of points $\setA = \{{\bf a}_1, \ldots, {\bf a}_n\}$, we denote by $\VDM (\setA,\setP)$  the Vandermonde matrix $\VDM (\setA,\setP) = \bmx p_j({\bf a}_i) \emx_{i,j=1}^{n,n}$ and by $\vdm (\setA,\setP) = \det \VDM (\setA,\setP)$
its determinant. We should note that $\VDM (\setA,\setP)$ is defined uniquely only if a specific order of the elements of $\setA$  and $\setP$ is fixed. However, we are mainly interested in the absolute value of $\vdm (\setA,\setP)$, and therefore, the particular order does not matter. For convenience, we also use the notation $\vdm (\setA,P) = \vdm(\setA, \{P_{i,j}(\bfz)\}_{i,j=1}^{m,n})$ and $\VDM (\setA,P) = \VDM(\setA, \{P_{i,j} (\bfz)\}_{i,j=1}^{m,n})$ for a polynomial matrix $P \in \bbR^{m\times n}[\bfz]$.

\section{Vandermonde determinants whose variables separate}
\subsection{The main result}
\begin{proposition}\label{prop:vdm_separable}
Assume that $m,n \ge 1$ are integers such that $n = m$ or $n = m+1$. Let  $P \in \bbR^{m\times n}[x]$, $Q \in \bbR^{m\times n}[y]$ be two polynomial matrices of the form
\begin{equation}\label{eq:stmat}
Q(y) = 
\bmx
1 & *      &  \dots & \dots         & *     \\
1 & q_1(y) &  *      &  \dots       & *      \\
\vdots & \vdots &  \ddots & \ddots  & \vdots \\
1 & q_1(y) &  \dots  & q_{n-1}(y)   & *     
\emx, \quad
P(x) = 
\bmx
*      & 1 &  1 & \dots    & 1 \\
*      & *      &  p_1(x) & \dots    & p_1(x)      \\
\vdots &        &  \ddots & \ddots     & \vdots \\
*      & *      &  \dots  & *        & p_{n-1}(x)     
\emx,
\end{equation}
i.e., the lower triangular block of $Q$ (including the main diagonal) has constant columns, and the upper triangular block of $P$ (excluding the main diagonal) has constant rows, and $*$ denote arbitrary (not necessarily equal to each other) polynomials.
Also assume that 
\[
\mbox{all }p_{j}, q_j \mbox{ are monic, } \quad \deg ({p_j}) = \deg ({q_j}) = j.
\]
Then for $\setX = \{x_1, \ldots, x_m\}$ and $\setY= \{y_1, \ldots, y_n\} $, the following equality holds:
\begin{equation}\label{eq:vdm_fac}
\vdm(\setX \times \setY,P \circ Q) = \pm\left(\prod\limits_{j=1}^{n} \vdm(\setX,P_{:,j}) \right) \cdot \left(\prod\limits_{i=1}^{m} \vdm(\setY, Q_{i,:}) \right).
\end{equation}
\end{proposition}
{\bf Remark}. In \eqref{eq:stmat} we show the matrices for $n=m+1$. The matrices for $n=m$ can be obtained by deleting the last column of each matrix in \eqref{eq:stmat}.
\begin{proof}
We prove the proposition by induction. First we consider the case $m=1$. For $n=1$, \eqref{eq:vdm_fac} is trivial. For $n = m+1 = 2$, consider $P(x) = \bmx a_1(x) & a_2(x)\emx$, $Q(y) = \bmx b_1(y) & b_2(y)\emx$, $\setX = \{x_1\}$ and $\setY= \{y_1, y_2\}$. Then we have that
\[
\vdm(\setX \times \setY,P \circ Q) = \pm a_1(x_1) a_2(x_1) \det \bmx 
b_1(y_1) & b_2(y_1) \\
b_1(y_2) & b_2(y_2) 
\emx.
\]
Now we assume that \eqref{eq:vdm_fac} holds for $(m,n) = (k,k)$ and  we prove it for $(m,n) = (k,k+1)$. Although we prove only the induction step $(k,k) \to (k,k+1)$, the same derivations (almost without changes) hold for the step $(k,k+1) \to (k+1,k+1)$. 

Denote $N_j = \VDM(\setX,P_{:,j})$. Then $\vdm( \setX \times \setY,P \circ Q) $ can be written as 
\begin{eqnarray}\nonumber
&&\vdm(\setX \times \setY, P \circ Q)   \\\nonumber
&& =  
\pm\det
\left[
\begin{array}{c|ccc}
N_1 & N_{2} \diag(Q_{:,2}(y_1)) & \cdots &  N_{m} \diag(Q_{:,n}(y_1))  \\\hline
N_1 & N_{2} \diag(Q_{:,2}(y_2)) & \cdots &  N_{m} \diag(Q_{:,n}(y_2))   \\
\vdots & \vdots & & \vdots\\
N_1 & N_{2} \diag(Q_{:,2}(y_n)) & \cdots &  N_{m} \diag(Q_{:,n}(y_n))   \\
\end{array}
\right]  \\ \label{eq:block_form}
&& =
\pm \det(N_1)\cdot
\det
\left[\begin{array}{c|ccc}
I_m   & N_{2} \diag(Q_{:,2}(y_1)) & \cdots &  N_{n} \diag(Q_{:,n}(y_1))  \\\hline
I_m   & N_{2} \diag(Q_{:,2}(y_2)) & \cdots &  N_{n} \diag(Q_{:,n}(y_2))   \\
\vdots & \vdots & & \vdots\\
I_m   & N_{2} \diag(Q_{:,2}(y_n)) & \cdots &  N_{n} \diag(Q_{:,n}(y_n))   \\
\end{array}\right].
\end{eqnarray} 
Note that in \cite{Bos.2009.OTV} a different block representation was used. That different representation was an obstacle to derive the proof of \cite[Lemma 1]{Bos.2009.OTV}.

By applying the Schur complement formula to the block matrix in \eqref{eq:block_form}, we have that
\begin{eqnarray}\nonumber
&&
\vdm(\setX \times \setY, P \circ Q) \\
\nonumber
 &&=
\pm\det(N_1) \det
\bmx
N_{2} \diag(Q_{:,2}(y_2) - Q_{:,2}(y_1)) & \cdots &  N_{n} \diag(Q_{:,n}(y_2)-Q_{:,n}(y_1))   \\
 \vdots & & \vdots\\
N_{2} \diag(Q_{:,2}(y_n) - Q_{:,2}(y_1)) & \cdots &  N_{n} \diag(Q_{:,n}(y_n)-Q_{:,n}(y_1))   \\
\emx \\ \nonumber
&&=
\pm\det(N_1) \left(\prod\limits_{j=2}^n(y_j-y_1)\right)^m 
\det \bmx
N_{2} \diag(\widetilde{Q}_{:,1}(y_2)) & \cdots &  N_{n} \diag(\widetilde{Q}_{:,n-1}(y_2))   \\
 \vdots & & \vdots\\
N_{2} \diag(\widetilde{Q}_{:,1}(y_n)) & \cdots &  N_{n} \diag(\widetilde{Q}_{:,n-1}(y_n))   \\
\emx, \\ \label{eq:det_reduction}
&& =
\pm\vdm(\setX,P_{:,1}) \left(\prod\limits_{j=2}^n(y_j-y_1)\right)^m 
\vdm\left(\setX \times \widetilde{\setY},P_{:, 2:n} \circ \widetilde{Q}\right).
\end{eqnarray}
where $\widetilde{\setY} := \{y_2,\ldots, y_n\}$ and $\widetilde{Q} \in \bbR^{m \times (n-1)}[y]$ is the polynomial matrix defined as
\[
\widetilde{Q}(y) := \frac{Q_{:,2:n} (y) - Q_{:,2:n} (y_1)}{y-y_1}.
\] 
Note that the matrix $\widetilde{Q}$ has the form
\begin{equation*}
\widetilde{Q}(y) = 
\bmx
\widetilde{*}   & \widetilde{*}      &  \dots & \dots       & \widetilde{*}     \\
1 & \widetilde{*} &  \dots      &   \dots       & \widetilde{*}      \\
1 & \widetilde{q}_1(y) &  \ddots      &          & \widetilde{*}      \\
\vdots & \vdots &  \ddots & \ddots   & \vdots \\
1 & \widetilde{q}_1(y) &  \dots  & \widetilde{q}_{n-2}(y)   & \widetilde{*}     
\emx, \quad
\end{equation*}
where $\widetilde{q}_j(y) := \frac{q_{j+1}(y) - q_{j+1}(y_1)}{y-y_1}$, and hence $\deg (\widetilde{q}_j) = j$ and $\widetilde{q}_j$ is monic. Therefore, we can interchange the variables $x$ and $y$, transpose the matrices $P_{:,2:n}$ and $\widetilde{Q}$ and apply the induction assumption to \eqref{eq:det_reduction}. Formally, for 
\[
P'(x') := (\widetilde{Q}(x'))^{\top} \in \bbR^{k\times k}[x'], \quad Q'(y') := (P_{:,2:n}(y'))^{\top}\in \bbR^{k \times k}[y'], \quad \setX' := \widetilde{\setY},\quad \setY' := \setX,
\]
the equality \eqref{eq:vdm_fac} takes place by the induction assumption, and we have that 
\[
\vdm\left(\setX \times \widetilde{\setY},P_{:, 2:n} \circ \widetilde{Q}\right) 
= \pm\left(\prod\limits_{j=2}^{n} \vdm( \setX,P_{:,j}) \right)  \left(\prod\limits_{i=1}^{m} \vdm(\widetilde{\setY},\widetilde{Q}_{i,:}) \right).
\]
Hence, from \eqref{eq:det_reduction} we have that
\[
\vdm(\setX \times \setY,P \circ Q) =
\pm\left(\prod\limits_{j=1}^{n} \vdm( \setX,P_{:,j}) \right)  \left(\prod\limits_{j=2}^n(y_j-y_1)\right)^m  \left(\prod\limits_{i=1}^{m} \vdm(\widetilde{\setY},\widetilde{Q}_{i,:}) \right).
\]
Then the equality \eqref{eq:vdm_fac} will hold if for all  $i = 1,\ldots, m$ the following equality holds: 
\begin{equation}\label{eq:last}
\pm\left( \prod\limits_{j=2}^n(y_j-y_1) \right) \vdm(\widetilde{\setY}, \widetilde{Q}_{i,:}) = \vdm(\setY,Q_{i,:}).
\end{equation}
Consider a row vector polynomial $A(y) = \bmx 1 & a_1(y) &\cdots & a_{n-1}(y)  \emx$. Then,
\begin{eqnarray*}
\vdm(\setY, A)&& =  
\pm\det
\left[
\begin{array}{c|ccc}
1 & a_1(y_1) & \cdots & a_{n-1}(y_1)  \\\hline
1 & a_1(y_2) & \cdots & a_{n-1}(y_2)  \\
\vdots & \vdots & & \vdots\\
1 & a_1(y_n) & \cdots & a_{n-1}(y_n) \\
\end{array}
\right]\\\nonumber
&& =  
\pm\det
\bmx
a_1(y_2) - a_1(y_1) & \cdots & a_{n-1}(y_2) - a_{n-1}(y_1) \\
 \vdots & & \vdots\\
a_1(y_n) - a_1(y_1) & \cdots & a_{n-1}(y_n) - a_{n-1}(y_1) \\
\emx
\\ \nonumber
&& = \pm\left(\prod\limits_{j=2}^n(y_j-y_1)\right)  \vdm(\widetilde{\setY},\widetilde{A}), \\\nonumber
\end{eqnarray*}
where 
\begin{equation}\label{eq:red_op}
\widetilde{A}(y) := \frac{\bmx
a_1(y) - a_1(y_1) & \cdots & a_{n-1}(y)  - a_{n-1}(y_1) \emx}{y-y_1}.
\end{equation}
This proves \eqref{eq:last}.
\end{proof}

\subsection{Discussion}
Consider a case which is simpler than that of Proposition~\ref{prop:vdm_separable}. Let $\setX = \{x_1, \ldots, x_m\}$ and $\setY = \{y_1,\ldots, y_n\}$.
Let $P(x) \in \bbR^{m \times n}[x]$ and $Q(y) \in \bbR^{m \times n}$ be given by
\begin{equation}\label{eq:stmat_rank_one}
P(x) = 
\bmx
p_1(x) & p_1(x) &  \dots  & p_{1}(x)     \\
p_2(x) & p_2(x) &  \dots  & p_{2}(x)  \\
\vdots & \vdots &         &  \vdots   \\
p_m(x) & p_m(x) &  \dots  & p_{m}(x)        
\emx,
\quad
Q(y) = 
\bmx
q_1(y) & q_2(y) &  \dots  & q_{n}(y)     \\
q_1(y) & q_2(y) &  \dots  & q_{n}(y)  \\
\vdots & \vdots &         &  \vdots   \\
q_1(y) & q_2(y) &  \dots  & q_{n}(y)        
\emx.
\end{equation}
Then we have that 
\[
(P \circ Q) (x,y) = 
\bmx
p_1(x)   \\
p_2(x) \\
\vdots  \\
p_m(x)    
\emx
\bmx
q_1(y) & q_2(y) &  \dots  & q_{n}(y) 
\emx,
\]
and therefore 
\begin{equation}\label{eq:vdm_rank_one}
\vdm (\setX \times \setY, P \circ Q) = \pm \det\Big( \VDM(\setX, \{p_1, \ldots, p_m\}) \otimes \VDM(\setY, \{q_1, \ldots, q_n\})\Big).
\end{equation}
By properties of the Kronecker product \cite{Horn.1991.TMA}, \eqref{eq:vdm_rank_one} can be rewritten as
\begin{equation}\label{eq:vdm_rank_one2}
\vdm (\setX \times \setY, P \circ Q) = \pm (\vdm(\setX, \{p_1, \ldots, p_m\}))^n (\vdm(\setY, \{q_1, \ldots, q_n\}))^m.
\end{equation}
An example of \eqref{eq:stmat_rank_one} is
\[
P(x) = 
\bmx
1 & 1 &  \dots  & 1     \\
x & x &  \dots  & x \\
\vdots & \vdots &         &  \vdots   \\
x^{m-1} & x^{m-1} &  \dots  & x^{m-1}
\emx,
\quad
Q(y) = 
\bmx
1 & y &  \dots  & y^{n-1}\\
1 & y &  \dots  & y^{n-1}  \\
\vdots & \vdots &        &  \vdots   \\
1 & y &  \dots  & y^{n-1}        
\emx,
\]
where
\[
\vdm (\setX \times \setY, P \circ Q) = \pm \left(\prod_{1\le i <j\le m} (x_i - x_j)\right)^n \left(\prod_{1\le i <j\le n} (y_i - y_j)\right)^m.
\]
It is easy to see, that \eqref{eq:vdm_rank_one2} is an equivalent of \eqref{eq:vdm_fac} for matrices of the form \eqref{eq:stmat_rank_one}. Thus, Proposition~\ref{prop:vdm_separable} can be interpreted as extension of the factorization property  from a special case of rank-one matrices \eqref{eq:stmat_rank_one} to a more general class \eqref{eq:stmat}. 

For the class of matrices \eqref{eq:stmat},  any submatrix that is contained in the upper triangle of $P(x)$ has rank $1$. The same holds for any submatrix contained in the lower triangle of $Q(y)$. Therefore, the matrix $P \circ Q$ is the Hadamard product of a \emph{lower semiseparable} matrix $Q(y)$ and an \emph{upper semiseparable} matrix $P(x)$ (these matrices are also called Hessenberg-like matrices in \cite[Ch. 8]{Vandebril.2007.MCS}). However, the relation to semiseparable matrices was not used in the proof of Proposition~\ref{prop:vdm_separable}.

\section{Application to Padua and Padua-like points}
\subsection{Main definitions}
We define a class of points that distribute as the $n$-degree Padua points \eqref{pd1}. For simplicity, we consider $n$ even (the case $n$ odd is similar). 
The $n$-degree \textit{Padua-like points} $\calA_{n}$ are defined as a union of two grids 
\[
\calA_{n} := \calA^{o}_{n} \cup \calA^{e}_{n},
\]
where
\[
\calA^{o}_{n} := \{(x_{2i+1}, y_{2j+1})\,|\, 0 \le i \le \frac{n}{2}, 0 \le j \le \frac{n}{2}\}
\]
and
\[
\calA^{e}_{n} := \{(x_{2i}, y_{2j})\,|\, 1 \le i \le \frac{n}{2}, 1 \le j \le \frac{n}{2} +1\},
\]
and $\{x_{i}\}^{n+1}_{i=1}$, $\{y_{j}\}^{n+2}_{j=1}$ are distinct sets of points. 

The Padua points \eqref{pd1} are a special case of Padua-like points, with
\[
\calA^{o}_{n} =  C_{n+1}^\mathrm{o} \times C_{n+2}^\mathrm{o}, \quad  \calA^{e}_{n} = C_{n+1}^\mathrm{e} \times C_{n+2}^\mathrm{e},
\]
and  $\{x_{i}\}^{n+1}_{i=1} =  C_{n+1}$, $\{y_{j}\}^{n+2}_{j=1} =  C_{n+2}$.

We are interested in expressing the Vandermonde determinant $\vdm(\calA_n,\calB_n)$, where
\[
\calB_n := \{x^\alpha y^\beta\,|\, \alpha+ \beta \le n; \alpha, \beta \ge 0\}
\]
is the set of all monomials of degree $\le n$. We  also define the square set of monomials
\[
\calT_n := \{x^\alpha y^\beta\,|\, 0 \le \alpha, \beta \le n \}.
\]

\subsection{Vandermonde determinant for Padua-like points}
It is easy to show  (see \cite{Bos.2009.OTV}) that
\[
\vdm(\calA_n,\calB_n) = \pm \vdm(\calA_n,\calT_{\frac{n}{2}} \cup \calT^{e}_{\frac{n}{2}} ),
\]
where 
\[
\calT^{e}_{\frac{n}{2}} := (a(x)\calB_{\frac{n}{2}- 1}) \cup (b(y) \calB_{\frac{n}{2} - 1}),
\]
and $a(x)$, $b(y)$ are the annihilating polynomials of the points $\calA^{o}_{n}$, that is
\[
a(x) = \prod_{i=0}^{\frac{n}{2}} (x-x_{2i+1}), \quad
b(y) = \prod_{j=0}^{\frac{n}{2}} (y-y_{2j+1}).
\]
By construction,  all the elements of $\calT^{e}_{\frac{n}{2}}$ vanish on $\calA^{o}_{n}$, which allows us to split the Vandermonde determinant into a product of two determinants:
\begin{eqnarray}\nonumber
\vdm(\calA_n, \calB_n) &= & \pm
\det \bmx
\VDM(\calA^{o}_{n}, \calT_{\frac{n}{2}}) & 0 \\
\VDM(\calA^{e}_{n}, \calT_{\frac{n}{2}}) & \VDM(\calA^{e}_{n}, \calT^{e}_{\frac{n}{2}}) \\
\emx
 \\\label{eq:padua_det}
& = &\pm \vdm(\calA^{o}_{n}, \calT_{\frac{n}{2}}) \times \vdm(\calA^{e}_{n}, \calT^{e}_{\frac{n}{2}}).
\end{eqnarray}
We consider the computation of the second factor in \eqref{eq:padua_det}. For this, we define two polynomial matrices of size $(\frac{n}{2}) \times (\frac{n}{2}+1)$
\begin{equation}\label{eq:stmat_padua}
\begin{split}
&Q(y) = 
\bmx
1      & b(y)    &  y b(y) & \dots                 & y^{\frac{n}{2} - 1} b(y)  \\
1      & y       &  b(y)   & \ddots               & \vdots      \\
\vdots & \vdots  &  \ddots & \ddots                & y b(y) \\
1      & y       & \dots   & y^{\frac{n}{2} - 1}   & b(y)     
\emx, \\
&P(x) = 
\bmx
a(x)   & 1      &  1      & \dots    & 1 \\
xa(x)   & a(x)   &  x      & \dots    & x      \\
\vdots &  \ddots      &  \ddots & \ddots   & \vdots \\
x^{\frac{n}{2} - 1}a(x)    & \dots     &  xa(x)  & a(x)      & x^{\frac{n}{2} - 1}     
\emx.
\end{split}
\end{equation}
It is easy to see that $\calT^{e}_{\frac{n}{2}}$ consists of the entries of $P \circ Q$,
and therefore 
\[
\vdm(\calA^{e}_{n}, \calT^{e}_{\frac{n}{2}}) = \vdm (\setX \times \setY, P \circ Q),
\]
for $\setX = \{x_{2i}\,|\,1\le i \le \frac{n}{2}\}$ and $\setY = \{y_{2i}\,|\,1\le i \le \frac{n}{2}+1\}$. Since the matrices $P(x)$ and $Q(y)$ satisfy the condition \eqref{eq:stmat}, Proposition~\ref{prop:vdm_separable} can be applied.
\begin{corollary}
\begin{equation}\label{eq:vdm_padua}
 \vdm(\calA^{e}_{n}, \calT^{e}_{\frac{n}{2}}) = \pm\left(\prod\limits_{j=1}^{\frac{n}{2}+1} \vdm(\setX,P_{:,j}) \right) \cdot \left(\prod\limits_{i=1}^{\frac{n}{2}} \vdm(\setY,Q_{i,:})\right).
\end{equation}
\end{corollary}

\subsection{Possible simplifications for inner determinants}
Note that we still need to compute the inner determinants in \eqref{eq:vdm_padua}, which have the form $\vdm(\setX, P_{:,j})$, with $\setX = \{x_1, \ldots, x_m\}$, $1 \le j \le m + 1$, and
\[
P_{:,j} (x) := \bmx 1 & x & \dots & x^{j-2} & a(x) &xa(x) & \dots &x^{m-j} a(x) \emx^{\top}.
\]
In the extreme cases
\[
\begin{split}
j = m+1: &\quad P_{:,m+1}(x) := \bmx 1 & x & \dots & x^{m-1}  \emx^{\top},\\
j = 1: & \quad P_{:,1} (x) :=  \bmx  a(x) &xa(x) & \dots &x^{m-1} a(x) \emx^{\top} = a(x) P_{:,m+1} (x),
\end{split}
\]
the determinant $\vdm(\setX, P_{:,j}(x))$ can be computed explicitly, as it was done in \cite[Lemma 3]{Bos.2009.OTV}. We will try to exploit \eqref{eq:last} to simplify the expression. First, we note that the successive application of the operation  \eqref{eq:red_op} leads to
\[
\frac{\frac{a(x) - a(x_1)}{x-x_1} - \frac{a(x_2) - a(x_1)}{x_2-x_1}}{x-x_2} =
\frac{a(x)  - \frac{x(a(x_2) - a(x_1)) -x_1a(x_2) + a(x_1) x_2}{x_2-x_1}}{(x-x_2) (x-x_1)}
\]
Therefore, the result of $\ell$ successive applications of \eqref{eq:red_op} to a polynomial $a(x)$ is equal to
\[
\widetilde{a}^{(\ell)}(x) = \frac{a(x) - p(x; x_1, \ldots, x_\ell)}{\prod\limits_{i=1}^\ell (x-x_i)},
\]
where $p(x; x_1, \ldots, x_\ell)$ is the $(\ell-1)$-th degree interpolating polynomial of $a(x)$ at $\{x_1, \ldots, x_\ell\}$. Then the following proposition holds.
\begin{proposition}\label{prop:simplify_1d}
For $2 \le j < (m-1)/2$ we have that
\[
\vdm(\setX, P_{:,j}) = 
\pm \left(\prod\limits_{i=1}^{j-1} \prod\limits_{k=i+1}^{m} (x_i - x_k)  \right)
\vdm(\{x_{j}, \ldots, x_{m}\}, {\setP}^{(j)}),
\]
where
\[
{\setP}^{(j)} = 
\{\widetilde{a}^{(j-1)}(x), \ldots, \widetilde{a}^{(1)}(x), a(x), x a(x), \ldots, x^{m-2j+1}a(x)\}.
\]
\end{proposition}
\begin{proof}
Consider a polynomial vector
\[
A(x) = \bmx \widetilde{a}^{(p)}(x) & \ldots & \widetilde{a}^{(1)} (x) & 1 &  \cdots &  x^{s} &  a(x) & x a(x) & \ldots & x^{q} a(x) \emx \in \bbR^{1\times M}[x],
\]
and apply the operation  \eqref{eq:red_op} to it (for the set of points $\setX = \{x_1,\ldots, x_M\}$):
\[
\widetilde{A}(x) =
\bmx
\widetilde{a}^{(p+1)}(x) & \ldots & \widetilde{a}^{(2)}(x) &
\widetilde{x} & \cdots &  \widetilde{x^{s}} &  \widetilde{a(x)} & \widetilde{xa(x)}  & \cdots & \widetilde{x^q a(x)} \emx.
\]
Therefore,
\[
\widetilde{x^{k}} =  \frac{x^{k} - x_1^{k}}{x-x_1} = x^{k-1} + x_1 x^{k-2} + \cdots + x_1^{k-1},
\]
and for $k \ge 1$,
\[
\widetilde{x^k a(x)} =  \frac{x^{k} a(x) - x_1^{k} a(x_1)}{x-x_1} = a(x) \frac{x^{k} - x_1^{k}}{x-x_1} + x_1^{k} \widetilde{a}(x).
\]
Hence, 
\[
\widetilde{A}(x) = B(x)
\bmx
1 & * & \cdots & * \\
0 & 1 & \cdots & \vdots \\
\vdots & \ddots & \ddots & * \\
0 & \cdots & 0 & 1
\emx,
\]
where
\[
B(x) =
\bmx
\widetilde{a}^{(p+1)}(x) & \ldots & \widetilde{a}^{(2)}(x) & \widetilde{a}^{(1)}(x) &
1 & \cdots &  x^{s-1} &   a(x)  & \cdots & x^{q-1} a(x)
\emx,
\]
Therefore,
\[
\begin{split}
\vdm (\{x_1,\ldots, x_M\}, A) &= 
\pm\left(\prod\limits_{k=2}^{M} (x_1 - x_k) \right) \vdm (\{x_2,\ldots, x_M\}, \widetilde{A})\\
&= 
\pm\left(\prod\limits_{k=2}^{M} (x_1 - x_k) \right) \vdm (\{x_2,\ldots, x_M\}, B).
\end{split}
\]
The rest of the proof follows by applying recursively the same argument, since the polynomial vector $B$ is of the same form as the polynomial vector $A$.
\end{proof}
We note that Proposition~\ref{prop:simplify_1d} can be also extended to handle the case $j \ge (m-1)/2$. We also note that Proposition~\ref{prop:simplify_1d} probably does not simplify much the expressions for the determinants of Padua-like points, since we still need to compute $\vdm(\{x_{j}, \ldots, x_{m}\}, {\setP}^{(j)})$.

\vskip 0.2in
{\bf Acknowledgments}. The first author is supported by the Progetto d'Ateneo ``Multivariate approximation with application to image reconstruction'' of the University of Padova (PRAT12).  The second author is supported by the European Research Council under the European Union's Seventh Framework Programme (FP7/2007-2013) / ERC Grant agreement No. 258581 ``Structured low-rank approximation: Theory, algorithms, and applications''.

\vskip 0.3in
\bibliographystyle{alpha}
\bibliography{vdm}

\newcommand{\etalchar}[1]{$^{#1}$}
\begin{thebibliography}{BMVX07}

\bibitem[BCM{\etalchar{+}}06]{Bos.2006.BLI}
L.~Bos, M.~Caliari, S.~De Marchi, M.~Vianello, and Y.~Xu.
\newblock Bivariate {L}agrange interpolation at the {P}adua points: the
  generating curve approach.
\newblock {\em Journal of Approximation Theory}, 143:15--25, 2006.

\bibitem[BMVX07]{Bos.2007.BLI}
L.~Bos, S.~De Marchi, M.~Vianello, and Y.~Xu.
\newblock Bivariate {L}agrange interpolation at the {P}adua points: the ideal
  theory approach.
\newblock {\em Numerische Mathematik}, 108(1):43--57, 2007.

\bibitem[BMW09]{Bos.2009.OTV}
L.~Bos, S.~De Marchi, and S.~Waldron.
\newblock On the {V}andermonde determinant of {P}adua-like points.
\newblock {\em Dolomites Research Notes on Approximation}, 2:1--15, 2009.

\bibitem[HJ91]{Horn.1991.TMA}
R.~A. Horn and C.~R Johnson.
\newblock {\em Topics in Matrix Analysis}.
\newblock Cambridge University Press, 1991.

\bibitem[VBM07]{Vandebril.2007.MCS}
R.~Vandebril, M.~Van Barel, and N.~Mastronardi.
\newblock {\em Matrix Computations \& Semiseparable Matrices I: Linear
  Systems}.
\newblock Johns Hopkins University Press, 2007.

\end{thebibliography}

\end{document}